\documentclass[]{amsart}
\usepackage{amssymb,mathdots}
\usepackage{mathrsfs}
\usepackage[applemac]{inputenc}
\usepackage{amsmath, graphicx}
\usepackage{wrapfig}
\usepackage{enumerate}
\usepackage{url}

\def\scaledpicture#1by#2(#3scaled#4){{
\dimen0=#1  \dimen1=#2
\divide\dimen0 by 1000 \multiply\dimen0 by #4
\divide\dimen1 by 1000 \multiply\dimen1 by #4
\picture \dimen0 by \dimen1 (#3 scaled #4)}}
\def\dfigure#1by#2(#3scaled#4offset#5:#6)
    {\medskip
     \vglue 2mm minus 2mm
     $$
       \hbox{
         \hglue#5
         {\scaledpicture #1 by #2 (#3 scaled #4)}
       }
     $$
     \par\goodbreak
     \vglue 2mm minus 2mm
     \medskip}

\def\qmod#1#2{{\hbox{}^{\displaystyle{#1}}}\!\big/\!\hbox{}_{
\displaystyle{#2}}}

\def\resto#1#2{{
#1\hskip 0.4ex\vline_{\hskip 0.2ex\raisebox{-0,2ex}
{{${\scriptstyle #2}$}}}}}

\def\C{{\mathbb C}}

\def\N{{\mathbb N}}

\def\P{{\mathbb P}}

\def\R{{\mathbb R}}
\def\Z{{\mathbb Z}}

\def\qed {\hfill\vrule height6pt width6pt depth0pt \smallskip}
\def\map{\longrightarrow}
\def\textmap#1{\mathop{\vbox{\ialign{
                                  ##\crcr
      ${\scriptstyle\hfil\;\;#1\;\;\hfil}$\crcr
      \noalign{\kern 1pt\nointerlineskip}
      \rightarrowfill\crcr}}\;}}
\def\bigtextmap#1{\mathop{\vbox{\ialign{
                                  ##\crcr
      ${\hfil\;\;#1\;\;\hfil}$\crcr
      \noalign{\kern 1pt\nointerlineskip}
      \rightarrowfill\crcr}}\;}}
      
\newcommand{\cal}{\mathcal}
\def\textlmap#1{\mathop{\vbox{\ialign{
                                  ##\crcr
      ${\scriptstyle\hfil\;\;#1\;\;\hfil}$\crcr
      \noalign{\kern-1pt\nointerlineskip}
      \leftarrowfill\crcr}}\;}}

\def\fg{{\mathfrak f}}

\def\sg{{\mathfrak s}}

\newtheorem{sz}{Satz}%[section]
\newtheorem{thry}[sz]{Theorem}
\newtheorem{pr}[sz]{Proposition}
\newtheorem{re}[sz]{Remark}
\newtheorem{co}[sz]{Corollary}
\newtheorem{dt}[sz]{Definition}
\newtheorem{lm}[sz]{Lemma}

\begin{document}
 
\def\tr{\mathrm {Tr}}
\def\End{\mathrm {End}}
\def\Aut{\mathrm {Aut}}
\def\Spin{\mathrm {Spin}}
\def\U{\mathrm{U}}
\def\SU{\mathrm {SU}}
\def\SO{\mathrm {SO}}
\def\PU{\mathrm {PU}}
\def\GL{\mathrm {GL}}
\def\spin{\mathrm {spin}}
\def\u{\mathrm {u}}
\def\su{\mathrm {su}}
\def\so{\mathrm {so}}
\def\pu{\mathrm {pu}}
\def\Pic{\mathrm {Pic}}
\def\Iso{\mathrm {Iso}}
\def\NS{\mathrm{NS}}
\def\deg{\mathrm {deg}}
\def\Hom{\mathrm{Hom}}
\def\Herm{\mathrm{Herm}}
\def\Vol{{\rm Vol}}
\def\pf{{\bf Proof: }}
\def\id{ \mathrm{id}}
\def\Im{\mathrm{Im}}
\def\im{\mathrm{im}}
\def\rk{\mathrm {rk}}
\def\ad{\mathrm {ad}}
%\def\coker{\mathrm{coker}}
%%%%%%%%%%%%%%%%%%%%%%%%%%
\def\spc{\mathrm{Spin}^c}
\def\U2{\mathrm{U(2)}}
\def\niq{=\kern-.18cm /\kern.08cm}
%%%%%%%%%%%%%%%%%%%
\def\Ad{\mathrm {Ad}}
\def\RSU{\R\mathrm{SU}}
\def\ad{{\rm ad}}
\def\dva{\bar\partial_A}
\def\da{\partial_A}
\def\p{{\rm p}}
\def\rk{\mathrm{rank}}
 \def\mult{\mathrm{mult}}

\def\oo{{\scriptstyle{\cal O}}}
\def\ooo{{\scriptscriptstyle{\cal O}}}
\def\subsetint{{\  {\subset}\hskip -2.45mm{\raisebox{.28ex}
{$\scriptscriptstyle\subset$}}\ }}

\title
{Intrinsic signs and  lower bounds in real algebraic geometry}
\author{Christian Okonek \& Andrei Teleman}
%
% \begin{abstract}  
 
% \end{abstract} 
%
\maketitle

\setcounter{section}{-1}

 \section{Introduction}

A classical result in complex algebraic geometry states that any smooth cubic surface in $\P^3_\C$ contains precisely 27 lines. It is natural to investigate the analogous  problem   in real algebraic geometry: how many {\it real} lines contains a {\it real} cubic surface in $\P^3_\R$. It is known that a real cubic surface in $\P^3_\R$ contains 27, 15, 7 or 3 real lines. A less known result  due to Segre states that on a real cubic surface there exists two kinds of real lines:   {\it elliptic} and {\it hyperbolic} lines. These two kinds of real lines are defined  in an  {\it intrinsic}  way, i.e.,  their definition does not depend on any choices of orientation  data. Using this separation into kinds Segre shows:
\begin{pr} (\cite{Se})\label{seg} Let $X$ be a smooth real cubic, $e$ the number of elliptic real lines, and $h$ the number of hyperbolic real lines on $X$. Then the following four cases can occur:
\begin{center}
\begin{tabular}{|c|c|c|c|}
\hline  
h&e&h+e&h-e\\ 
\hline
15&12&27&3\\
\hline
9&6&15&3\\
\hline
5&2&7&3\\\hline
3&0&3&3\\
\hline
\end{tabular}\ 
\end{center}
\end{pr}

Apparently the fact that the difference $h-e$ is always equal to 3  has not been noticed before, neither by Segre nor in later publications (see for instance \cite{BS})\footnote{V. Kharlamov has informed us that he together with Fintushin had noticed this fact while working on a forthcoming  article about real lines on projective hypersurfaces (see \cite{FK1}).}. Two important  facts can be noticed using this table:
\begin{enumerate}
\item There exists a non-trivial lower bound for the total number of real lines on a smooth real cubic surface. This lower bound is 3.\vspace{2mm}

Note that one has an obvious lower bound which follows by comparison with the complex situation. Indeed, this method yields an upper bound $h+e\leq 27$, but also the congruence $h+e\equiv  27$ (mod 2) which implies the lower bound $h+e\geq 1$.
\vspace{2mm}
\item The existence of two kinds of real lines, the definition of the two kinds being intrinsic, i.e., independent of any choices of orientation data.
\end{enumerate}
\vspace{2mm}

Starting from these remarks and inspired by the classical problem mentioned above, our article  has the following goals:
\begin{enumerate}
\item we explain a general principle which leads to lower bounds in real algebraic geometry,
\item we explain the reason  for the appearance of  intrinsic signs  in the classical  problem treated by Segre,  showing that the same phenomenon occurs in a large class of  enumerative problems in real algebraic geometry.
\item we illustrate these two principles in the enumerative problem of counting real lines in real hypersurfaces of degree $2m-3$ in $\P^m_\R$.   
\end{enumerate}
\vspace{2mm}
A short summary of our article folows:\\

The first section introduces the canonical Euler class of a real vector bundle. This class is defined for real vector bundles in an  intrinsic way, so  the choice of an orientation is not needed.  We also introduce relatively oriented bundles on smooth manifolds, and the concept of degree for {\it relatively} oriented bundles over closed manifolds. The essential point here is that relative orientability  is a much weaker condition than the orientability of both the bundle and the base manifold (which is the condition usually required  in the literature). At the end we treat the analogous concepts in the case of sphere bundles. \\

In the second section we prove a general lower estimate for the number of zeros  of a section $s$ with isolated zeros in a relatively oriented bundle $\xi$ over a closed manifold. Our result states that
$$\sum_{s\in Z(s)}\mult_x(s)\geq|\deg|(\xi) \ ,
$$
where the multiplicity $\mult_x(s)$ at an isolated zero is defined using algebraic methods.  It is important to note that the formula holds in the  ${\cal C}^\infty$ category and no transversality property is necessary. The result uses a classical inequality in singularity theory (see \cite{EL}) which compares the multiplicity of a smooth germ $(\R^n,0)\to(\R^n,0)$ with an isolated zero at $0\in\R^n$ with its Milnor degree at 0. Our result can be viewed as a global version of the inequality obtained in \cite{EL}.\\

The third section is dedicated to applications of the general theory to the  problem of counting real lines in a hypersurface $Z(f)\subset \P(V)$ defined by a homogeneous polynomial of degree $k:=2n-5$ on an $n$-dimensional real vector space $V$.  These lines correspond bijectively to the zeros of   the  section $s_f\in\Gamma(S^k(U^\vee))$ associated with $f$, where we denote  by $U$ the tautological plane bundle  of  the Grassmannian $G_2(V)$, and the section $s_f$ is obtained by restricting $f$ to the linear 2-planes $W\subset V$  which are the fibres of $U$. We start with a fundamental remark: the bundle $S^k(U^\vee)$ is {\it canonically} relatively oriented. This implies that  to any  regular zero of $s_f$ one can assign a sign  in an {\it intrinsic}  way. Note that both the Grassmannian $G_2(V)$ and the bundle $S^k(U^\vee)$ are non-orientable when $n$ is odd. For the canonical Euler class of $S^k(U^\vee)$ we obtain the simple formula
$$e_{S^k(U^\vee)}=1\cdot 3\cdot\hdots \cdot k \ .
$$
An important difficulty in proving this result is to control the canonical relative orientation of $S^k(U^\vee)$  in order to determine accurately the sign of this  Euler class. There is an important detail here which should be pointed out: the result depends on certain orientation conventions in linear algebra, which must be  fixed and used consistently; for instance one has to fix once for ever  the ``standard" orientation of the symmetric power $S^k(T)$ of an oriented vector space $T$, and the ``standard" orientation  of the morphism space $\Hom(T',T'')$ of two oriented vector spaces $T'$, $T''$.

The computation of this Euler class for $n=5$  has been obtained before in \cite{PSW};   in the general case $|e_{S^k(U^\vee)}|$ has recently been  computed independently by Finashin and Kharlamov  \cite{FK1}, \cite{FK2}. These authors interpret it as  the ``virtual number" of real lines.

As a corollary we obtain  the following lower estimate for the number of lines on the real hypersurface $Z(f)$,  under the (very weak) assumption that the set  ${\cal R}$ of real lines  is finite:
$$\sum_{l\in{\cal R}} \mult_l(s_f)\geq 1\cdot 3\cdot\hdots \cdot k \ .
$$
We have been  informed by V. Kharlamov that he is preparing a survey article together with Finashin in which they will prove a similar result (see \cite{FK2}).
\\

The fourth section deals with the intrinsic sign of a regular real line $l$ (regarded as zero of $s_f$) on a real hypersurface $Z(f)$ of degree $k:=2n-5$ in $\P(V)$, with respect to the canonical relative orientation of ${S^k(U^\vee)}$.  In order to do this we compute  explicitly the intrinsic derivative of the section $s_f$ at point in the zero locus, and give a general formula for its determinant. Writing $l=\P(W)$ for a 2-plane $W\subset V$ we see that if $W$ is a zero of $s_f$, then $f$ induces a   linear map
$$\fg:V/W\to \R[W]_{k-1}
$$
given by $\fg([v])(w):=f(v,\underbrace{w,\dots,w}_\textrm{$k-1$})$.   This map has an interesting geometric interpretation: its zero set $Z(\fg)\subset\P(W)\times\P(V/W)$ is the exceptional divisor of the blow up of $\P(W)\subset Z(f)$. The point is that $\fg$ determines the intrinsic derivative 
$$D_W(s_f):\Hom(W,V/W)=T_W(G_2(\R^{m+1}))\to\R[W]_k$$
via the formula $D_W(s_f)(\phi)(w)=k\fg(\phi(w))(w)$. If $W$ is a regular zero of $s_f$, then $\wedge^{n-2}\fg$ is a non-trivial element of $\Hom(\det (V/W),\wedge^{n-2}(\R[W]_{k-1}))$ and the corresponding point $[\wedge^{n-2}\fg]\in 
\P(\wedge^{n-2}(\R[W]_{k-1}))$ lies in the complement of a hyperquadric $Q\subset \P(\wedge^{n-2}(\R[W]_{k-1})$ whose equation can be written down explicitly. The complement of $Q$ has two connected components $\P(\wedge^{n-2}(\R[W]_{k-1})_\pm$, and the sign associated with $l=\P(W)$ is determined by the component to which $[\wedge^{n-2}\fg]$ belongs. In the case $n=4$, $k=3$ considered by Segre, we obtain an alternative geometric interpretation, which shows that the intrinsic signs which we assign to regular real lines using our  general formalism and the canonical relative orientation of the bundle $S^k(U^\vee)$ correspond to the two kinds (hyperbolic and elliptic) defined by Segre.

Note that   $W$ is a regular zero of $s_f$ if and only if $W\otimes\C$ is a regular zero of the corresponding section on the complex Grassmannian $G_2(V\otimes\C)$, and this happens iff the Grothendieck decomposition of the normal bundle of $\P(W\otimes\C)$ in  the complex hypersurface $Z(f^\C)$  is ${\cal O}_{\P(W\otimes\C)}(-1)^{\oplus(n-3)}$.
\\

We conclude with several remarks:
\vspace{2mm}\\ 
  Our  techniques can be applied to other interesting situations. We have in mind
\begin{enumerate}
\item the enumerative problem  for real lines on  real complete intersections, for instance on a complete intersection $Z(f,g)\subset \P^6_\R$, where $f$, $g$ are  polynomials of degree 3.
\item  applications in Real gauge theory on 2 and 4-dimensional manifolds:  We believe that similar phenomena (a priori lower estimates for the cardinality of certain 0-dimensional moduli spaces, and intrinsic signs associated with regular points in a 0-dimensional moduli space) also  occur in the infinite dimensional gauge theoretical framework. 
\end{enumerate}

Our results lead to the following natural open questions and problems:
\begin{enumerate}
\item  Further geometric interpretations of the sign of $\det(D_W(s_f))$. It would be interesting to have an explicit geometric interpretation of the sign of $\det(D_W(s_f))$ which uses only geometric properties of a hypersurface $Z(f)$ (or of a complete intersection) around a regular real line $l$.
\item Can one develop a wall crossing theory for the numbers $h$, $e$ of real lines with fixed intrinsic sign (1, respectively -1)? Note also that, according to the table obtained by Segre (see Proposition \ref{seg}) for cubics surfaces in $\P^3$ the number $e$ is always even.
\item Develop an orbifold version of the theory, allowing  to get lower bounds for the number of zeros of a section in a real bundle on a compact orbifold.
\end{enumerate}

\section{The canonical Euler class and the degree}

\subsection{The canonical Euler class}
Let $B$ be a topological space, and $\xi$ a real vector bundle of rank $r$ on $B$ whose projection map is $p:E\to B$. We denote by  $\oo_\xi$  the orientation coefficient systems of   $\xi$  regarded as sheaf  over $B$, locally isomorphic with $\underline{\Z}$.   The fibres of this sheaf are
$$ \oo_{\xi,x}=H_r(E_x,E_x\setminus\{0_x\};\Z)\ .
$$
The canonical Euler class of $\xi$ is a cohomology class $e_\xi\in H^r(B;\oo_\xi)$, which is defined in the following way: The Thom class $t_\xi\in H^r(E,E\setminus B;p^*(\oo_\xi))$ of $\xi$ is characterized by the condition that its restriction to any pair $(E_x,E_x\setminus \{0_{E_x}\})$ is the canonical generator of the cyclic group
$$H^r(E_x,E_x\setminus \{0_{E_x}\};\oo_{\xi,x})=\Hom(H_r(E_x,E_x\setminus\{0_x\};\Z),H_r(E_x,E_x\setminus\{0_x\};\Z))\ .$$
 Then one has the inclusions 
$$(B,\emptyset)\textmap{0_\xi}(E,\emptyset)\textmap{i_E}  (E,E\setminus B) \ ,
$$
and the canonical Euler class is defined \cite{EM} by:
$$e_\xi:= (i_E\circ 0_\xi)^*(t_\xi) $$

As the classical Euler class for oriented bundles, the canonical Euler class is multiplicative with respect to direct sums. More precisely
\begin{re}\label{W} Let $\xi_1$, $\xi_2$ two vector bundles of ranks  $r_1$, $r_2$ on $B$. Then there exists a canonical isomorphim $\oo_{\xi_1\oplus\xi_2}= \oo_{\xi_1}\otimes \oo_{\xi_2}$, and via this isomorphism one has
$$e_{\xi_1\oplus\xi_2}=e_{\xi_1}\cup e_{\xi_2}\ .
$$
\end{re}
{\ }\\
{\bf Example:} Let $\xi$ be an Euclidian   bundle of rank 2 on  $B$, i.e., $\xi$ is endowed with a reduction of the structure group from $\GL(2,\R)$ to $\mathrm{O}(2)$. Let $\pi:P\to B$ be the corresponding principal  $\mathrm{O}(2)$-bundle. Then one has natural isomorphisms
$$\det(\xi)=P\times_{\det}\R\ ,\ \oo(\xi)=P\times_{\det}\Z\ .
$$
The canonical Euler class $e_\xi$ introduced above coincides with the twisted first Chern class $\tilde c_1(\xi)\in H^2(B,P\times_{\det}\Z)$ defined in \cite{Fr}. Note that  an $\mathrm{O}(2)$-bundle can be regarded as a {\it twisted Hermitian line bundle} with $\oo(\xi)=P\times_{\det}\Z$  as twisting sheaf.
\\
\\
The canonical Euler class has the following important functoriality property: 
Let $f:B'\to B$ be a continuous map and $\xi'$ a rank $r$ bundle on $B'$ with projection  $p':E'\to B'$.  Let $F:E'\to E$  be a bundle map over $f$  (which is  fibrewise an isomorphism). Then $F$ defines in an obvious way a sheaf isomorphism $F_*: \oo_{\xi'}\to f^*(\oo_\xi)$, and 
\begin{equation}\label{func}F_*(e_{\xi'})=f^*(e_\xi)\ .
\end{equation}
In particular, suppose that $B=B'$, $\xi=\xi'$ and $F$ is a gauge transformation of $\xi$. Then $F_*$ is a sheaf automorphism of $\oo_\xi$ and     $F_*(e_\xi)=e_\xi$; taking $F=-\id_E$, we get $F_*=(-1)^r\id_{\oo_\xi}$. Hence we obtain the following remark which generalizes a well-known property of the classical Euler class \cite{MS}.

\begin{re}\label{oddrank} If $r$ is odd then $2e_\xi=0$.
\end{re}
If $B$ is paracompact we can use an inner product  to identify $\xi$ with its dual, and this identification is unique up to homotopy. Using the functoriality property we obtain
\begin{re}\label{dual} If $\xi$ is a real vector bundle of rank $r$ on a   paracompact basis $B$  then $e_\xi=e_{\xi^\vee}$   via the obvious identification $\oo_\xi=\oo_{\xi^\vee}$.  
\end{re}

Suppose now that $B$ is a closed, connected   manifold of dimension $d$, and let $\oo_B$ be its orientation sheaf. Its fibres are given by  $\oo_{B,x}=H_d(B,B\setminus\{x\};\Z)$. Applying the Poincaré duality isomorphism to $e_\xi$ one obtains a homology class %
$$D_B(e_\xi)\in H_{d-r}(B,\oo_B\otimes\oo_\xi)\ .$$ 

\subsection{Relatively oriented bundles} 

The bundle $\xi$ is called {\it relatively orientable}   when the tensor product $\oo_B\otimes\oo_\xi$ is isomorphic with the constant sheaf $\underline{\Z}$ over $B$. The choice of an isomorphism $\theta:\oo_B\otimes\oo_\xi\to\underline{\Z}$ then allows us to define a class $\theta_*D_{B}(e_\xi) \in H_{d-r}(B;\Z)$, which has the following important {\it geometric interpretation} in the differentiable case:

Let $s\in\Gamma(B,\xi)$ be a regular section of $\xi$, and denote by $Z(s)$ the (smooth) zero locus of $s$. Orient   $Z(s)\subset B$ using the identification 
$$\det(T_{Z(s)})=\det(\resto{T_B}{Z(s)}\otimes\det(N_{Z(s)})^\vee\ ,
$$
the isomorphism $N_{Z(s)}\simeq\resto{\xi}{Z(s)}$ defined by the intrinsic derivative $Ds$, and the trivialization $\theta$. Let $[Z(s)]_\theta$ be the corresponding fundamental class. Then
\begin{equation}\label{gi}\iota_*([Z(s)]_\theta)=\theta_*D_{B}(e_\xi)\in H_{d-r}(B;\Z)\ ,
\end{equation}
where $\iota:Z(s)\hookrightarrow B$ denotes the inclusion map. In particular, for $r=d$ we get  a well-defined invariant
$$\deg(\xi,\theta):=\nu(\theta_*D_{B}(e_\xi))\ ,
$$
where  $\nu$ is the standard augmentation map $H_0(B;\Z)\to\Z$. In this case the zero locus $Z(s)$ of a regular section  is an oriented 0-dimensional manifold, which can be written as $\sum_{x\in Z(s)} \epsilon_{x,\theta} x$, where
$$ \epsilon_{x,\theta}=\mathrm{sign}_\theta(\det(D_x(s))\ ,
$$
and $\mathrm{sign}_\theta$ denotes the sign computed using the orientation of the line $\det(T_x(B))^\vee\otimes\det(E_x)$ defined by $\theta_x$.

The geometric interpretation of formula (\ref{gi}) becomes   
\begin{equation}\label{gi2}
\deg(\xi,\theta)=\sum_{x\in Z(s)} \epsilon_{x,\theta}\ .
\end{equation}

This formula has an important generalization for sections with finitely many zeros. In order to state this result we introduce the local degree $\deg_x(s,\theta_x)$ of a section at an isolated zero with respect to a relative orientation $\theta_x$ at $x$. Let $\tau:\resto{E}{U}\to U\times\R^d$ a trivialization of $\xi$ on a small neighborhood of $x$. This trivialization and the natural orientation of $\R^d$ defines an orientation of $E_x$, and $\theta_x$ induces an isomorphism $\oo_{E_x}\to\oo_{B,x}$ hence a local orientation of $U$ around $x$.
The section $s$ corresponds to a map $\hat s:U\to\R^d$ with an isolated zero at $x$.  Since $U$ and $\R^d$ are oriented we have a well defined local degree  of $\hat s$ at $x$ \cite{EL}. It is easy to  see that this degree does not depend on the trivialisation $\tau$. We denote this local degree by $\deg_x(s,\theta_x)$.  As in the classical case one can prove:

\begin{lm} \label{ld} Let $\xi$ be a relatively orientable bundle of rank $d$ over a closed $d$-manifold $B$, and let $s\in \Gamma(B,\xi)$ be a section with isolated zeros. Let $\theta$ be a trivialisation of $\oo_B\otimes\oo_\xi$. Then
$$\deg(\xi,\theta)=\sum_{x\in Z(s)} \deg_{x}(s,\theta_x)\ .
$$
\end{lm}

 The absolute value $|\deg(\xi,\theta)|$ is well defined (independent of the choice of $\theta$) and will be denoted by $|\deg|(\xi)$. Therefore we have   for every relatively orientable bundle $\xi$  an invariant $|\deg|(\xi)\in\N$, which we call {\it the absolute degree of} $\xi$.
\\ \\
{\bf Example:} The tangent bundle $T_B$ is obviously {\it canonically relatively oriented}, because the tensor product $\oo_B^{\otimes 2}$ comes with a canonical trivialisation $\theta_{\rm can}$. Therefore, for any closed, connected    manifold $B$, one obtains a well defined integer invariant   
$$\deg(T_B):=\nu(\theta_{{\rm can} *} D_B (e_{T_B}))\ .
$$

This invariant can easily be identified with the Euler characteristic of $B$. This  is well known for orientable manifolds; in the non-orientable case it  can be seen using the orientation double cover of $B$, and  using the functoriality of the canonical Euler case to reduce the problem to the oriented case.

\subsection{Sections in sphere bundles}
 
 We denote by $E^+_B$ the space  over $B$ obtained by compactifying the fibres $E_x$ of $\xi$ with   points $\infty_x$. Let $\xi^+$  be the corresponding sphere bundle with projection $p^+:E^+_B\to B$. One has a natural Thom class $t^+(\xi)\in H^r(E^+_B,E^+_B\setminus B;(p^+)^*(\oo_\xi))$ whose restriction to any pair $(E_x,E_x\setminus \{0_{E_x}\})$ is the canonical generator of the cyclic group $H^r(E_x,E_x\setminus \{0_{E_x}\};\oo_{\xi,x})$.   Let $\sigma\in\pi_0(\Gamma(B,\xi^+))$ be a homotopy class of sections in $\xi^+$.  We define
$$e_{\xi,\sigma}:=s^*(t^+_\xi)\ \hbox{for }s\in \sigma\ ,
$$
where on the right $s$ is regarded as a map $(B,\emptyset)\to (E^+_B,\emptyset)\to (E^+_B,E^+_B\setminus B)$.  Note that in general the class $e_{\xi,\sigma}$ depends  on $\sigma$ in an essential way;   it is not an intrinsic  invariant of the bundle $\xi$ (see the example below). Suppose now that $B$ is a differentiable $d$-manifold and $\xi$ is a differentible vector bundle. When $\xi$ is relatively orientable,    and $s\in \Gamma(B,\xi^+)$  is  a regular section, we obtain again a similar geometric interpretation of the Euler class $e_{\xi,\sigma}$:
$$\iota_*([Z(s)]_\theta)=\theta_*D_{B}(e_{\xi,[s]})\in H_{d-r}(B;\Z) 
$$
In the case   $r=d$ we   put
$$
\deg(\xi,\theta,\sigma):=\nu(\theta_*D_{B}(e_{\xi,\sigma}))\ .
$$
The degree of a section in relatively oriented sphere bundles plays an important role in real gauge theory \cite{OT2}.
\\ \\
{\bf Example:} Let $\xi=B\times \R^r$ be the trivial bundle over $B$ with fibre $\R^r$ and $s\in\Gamma(B,\xi^+)={\cal C}^\infty(B,\mathrm{S}^r)$. In this case one obtains 
$$e_{\xi,\sigma}=s^*([\mathrm{S}^r]')\ ,
$$
where on the right $[\mathrm{S}^r]'$ denotes the fundamental class of $\mathrm{S}^r$ in cohomology. When $B$ is oriented and $r=d$ we have a canonical isomorphism $\theta_{\mathrm{can}}:\oo_B\otimes\oo_\xi\to\underline{\Z}$ and then
$$\deg(\xi,\theta_{\mathrm{can}},[s])=\deg(s)\ ,
$$
where $\deg(s)$ is defined by the equality $s^*([S^r]')=\deg(s)[B]'$. Hence $\deg(\xi,\theta_{\mathrm{can}},[s])$ coincides with the degree of the map $s:B\to S^r$. 

 In particular, as mentioned above, the Euler class $e_{\xi,\sigma}$  depends  on the homotopy class $\sigma$.

\section{Lower bounds  for the number of zeros of a section}

 Many enumerative problems in real algebraic geometry reduce to  counting the number of zeros of  a section $s\in\Gamma(B,\xi)$ in a  real algebraic bundle $\xi$  with   $\rk(\xi)=\dim(B)$. The necessary formalism can be developed more generally in the differentiable category, so from now we suppose that $B$ is a differentiable manifold and $\xi$ is a differentiable vector bundle on $B$. Let $s\in\Gamma(B,\xi)$ be  a section which is regular, so that   the zero locus $Z(s)$ of $s$ is a smooth 0-dimensional submanifold of $B$.  If $\xi$ is relatively orientable  we can choose a trivialisation $\theta$ of $\oo_B\otimes\oo_\xi$.  Then we have
$$\deg(\xi,\theta)=\sum_{x\in Z(s)} \epsilon_{x,\theta}\ ,
$$
with $\epsilon_{x,\theta}\in\{\pm 1\}$. Taking absolute values we obtain the following simple, but important
\begin{pr}\label{p1} Let $B$ be a   closed, connected differentiable $d$-manifold, $\xi$ a relatively orientable bundle of rank $r=d$  on $B$, and let $s\in\Gamma(B,\xi)$ be  a regular section. Then
\begin{equation}\label{lb}\#Z(s)\geq|\deg|(\xi)\ .
\end{equation}
\end{pr}

We now come  to the more general situation where $Z(s)$ is finite, but $s$ is not necessary transversal to the 0-section. In order  to generalize the estimate (\ref{lb}) we need a good notion of multiplicity of an isolated zero of $s$.

Denote by ${\cal C}^\infty_x(B,\R)$ the $\R$-algebra of germs at $x$ of differentiable functions defined around $x$, and by $\Gamma_x(B,\xi)$ the space of germs at $x$ of smooth sections of $\xi$ defined around $x$. A section $s$ defines an ideal $(s_x)\subset {\cal C}^\infty_x(B,\R)$ given by
$$(s_x):=\{\langle \lambda,s_x\rangle\ ,\ \lambda\in \Gamma_x(B,\xi^\vee)\}\ .
$$

When $\xi$ is trivial   $s$ can be regarded as an $\R^d$-valued smooth function vanishing at $x$, and then  $(s_x)$ is just the ideal generated by the germs at $x$ of the $n$ components  of $s$. Following  \cite{EL}  we define the $\R$-algebra
$$Q(s_x):=\qmod{{\cal C}^\infty_x(B,\R)}{(s_x)} \ .
$$
The dimension  
$$\mult_x(s):=\dim_\R(Q(s_x))\in\N^*\cup\{\infty\} 
$$
of this quotient is called the multiplicity of $s$ at $x$ (see \cite{EL} p. 22). The multiplicity is an important invariant of the germ which has been studied extensively in the literature. The germ $s_x$ is called finite if  $\mult_x(s)$ is finite. In this case  $x$ is an isolated zero  of $s$, and  the Taylor expansion with respect to a chart $h:(U,x)\to (V,0)$ around $x$ induces an  isomorphism
$$Q(s_x)\  {\simeq}\ \qmod{\R[[X_1,\dots,X_d]]  }{\langle \sg_{1,0},\dots,\sg_{d,0}\rangle }\ .
$$ 
Here   $\sg_{i,0}$ are the Taylor series of the components of $s_x$ with respect to  $h$ and a trivialization of  $\xi$ around $x$. If $B$, $\xi$ and $s$ are real analytic, then   the algebra $Q(s_x)$ can also be obtained using the $\R$-algebra of {\it convergent} power series (see \cite{EL} p. 31), i.e.,
$$Q(s_x)\ \simeq\ \qmod{\R\{ X_1,\dots,X_d\}  }{\langle \sg_{1,0},\dots,\sg_{d,0}\rangle }\ .
$$
 Using this remark, we obtain the following important complex geometric interpretation of the multiplicity of a real analytic germ: 
\begin{re} \label{multC} Suppose $B$, $\xi$ and $s$ are real analytic and $s_x$ is finite. Let  $s_{i}$ be the components of $s_x$ with respect to an analytic chart $h:(U,x)\to (V,0)$ and an analytic  trivialization of $\xi$  around $x$, and let $\tilde s_{i}$ be holomorphic  extensions of these components. Then the     complex space $Z(\tilde s_{1},\dots \tilde s_{d})$ is 0-dimensional at 0, and
$$\mult_x(s)=\dim_\C({\cal O}_{Z(\tilde s_{1},\dots \tilde s_{d}),0})=\mathrm{length}(Z(\tilde s_{1},\dots \tilde s_{d}),0)\ .
$$
\end{re}
The main result of \cite{EL} shows   
\begin{pr} Let $B$ be a differentiable $d$-manifold, $\xi$ a vector bundle of rank $d$, and  let $s$  be a section of $\xi$. Suppose  $x$ is an isolated zero of $s$ and $s_x$ is finite.  Then
$$\mathrm{mult}_{x}(s)\geq| \deg_{x}(s_x)|\ .$$
\end{pr}  
 
 Combining this proposition with Lemma \ref{ld} we get the following estimate 
$$|\deg|(\xi)=|\deg(\xi,\theta)|\leq\sum_{x\in Z(s)}| \deg_x(s,\theta_x)|\leq
\sum_{x\in Z(s)}\mathrm{mult}_{x}(s) \ .
$$

This proves
\begin{pr} \label{p2} Let $B$ be a closed, connected differentiable $d$-manifold, $\xi$ a relatively orientable vector bundle of rank $d$ on $B$, and  $s\in \Gamma(B,s)$   a smooth section with isolated zeros. Then
\begin{equation}\label{lban}
\sum_{x\in Z(s)} \mathrm{mult}_{x}(s)\geq|\deg| (\xi)\ .
\end{equation}
\end{pr} 

\subsection{Comparison results: real degree versus complex degree}

Let $(X,\tau)$ be a topological space endowed with an involution $\tau$. We recall that a Real vector bundle on $X$ is a complex vector bundle $E$ endowed with an anti-linear bundle   automorphism $\tilde\tau:E\to E$ which lifts $\tau$ and is also an involution. The involution $\tilde\tau$ is called Real structure on $E$. If $(E,\tilde\tau)$ is a Real vector bundle on    $(X,\tau)$  then the fixed point locus $E(\R):=E^{\tilde\tau}$ of $\tilde\tau$  is a real bundle (in the standard sense) on the fixed point locus $X(\R):=X^\tau$ of $\tau$.

If $X$ is a complex manifold, then a real structure on $X$ is an anti-holomorphic involution $\tau:X\to X$,  a Real holomorphic bundle on $X$ is a holomorphic bundle endowed with an anti-holomorphic Real  structure, and a Real section of a Real holomorphic bundle $(E,\tilde\tau)$ is a $\tilde\tau$-invariant holomorphic section $s$ of $E$.  Such a section induces a real analytic section $s(\R)$ of the bundle $E(\R)\to X(\R)$.

With this definitions we can state  the following comparison result relating  the degree of $E(\R)$ to the degree of $E$, and the multiplicites of the zeros of  a Real holomorphic section of a Real holomorphic vector bundle.

\begin{pr} Let $X$ be compact complex manifold endowed with real structure, and let $E\to X$ be a Real holomorphic vector bundle over $X$ with $\rk(E)=\dim(X)=n$. Suppose that the real vector bundle $E(\R)\to X(\R)$ is relatively orientable, and let $s$ be a Real holomorphic section of $E$ with finite zero locus $Z(s)$. Then
\begin{enumerate}
\item  $|\deg|(E(\R))\leq \sum_{z\in Z(s)\cap X(\R)}\mult_z(s)\leq\langle c_n(E),[X]\rangle\ ,$ 
\item  $|\deg|(E(\R))\equiv \langle c_n(E),[X]\rangle\   \hbox{ (mod 2)}\ ,$ 
\item $|\deg|(E(\R))\equiv \sum_{z\in Z(s)\cap X(\R)}\mult_z(s)\  \hbox{ (mod 2)}\ .
$
\end{enumerate}
\end{pr}
\begin{proof} (1) Applying Proposition \ref{p2} to the section $s(\R)\in\Gamma(X(\R)),E(\R))$, and using Remark \ref{multC} we obtain 
$$\langle c_n(E),[X]\rangle=\sum_{z\in Z(s)}\mult_z(s)=\sum_{z\in Z(s)\cap X(\R)}\mult_z(s)+ \sum_{z\in Z(s)\setminus X(\R)}\mult_z(s) \geq $$
$$
\geq |\deg|(E(\R))\ .
$$
\\
(2) We have $$\langle c_n(E),[X]\rangle\hbox{ (mod 2)}=\langle w_{2n}(E),[X]_2\rangle=\langle w_n(E(\R)),[X(\R)]_2\rangle=$$
$$=\deg(E(\R))\hbox{ (mod 2) },
$$ 
where the second equality follows from Theorem 6.4 in \cite{OT1}.
\\
\\
(3)  This follows from (2) and from the formula used in the proof of (1) taking into account that $\sum_{z\in Z(s)\setminus X(\R)}\mult_z(s)\in 2\N$.
\end{proof}
\section{Counting real lines in real hypersurfaces}
 
Let $V$ be an $n$-dimensional real vector space $V$,  $f\in S^k(V^\vee)$ a $k$-linear form on $V$.  The  associated  polynomial function  defines a section $\sigma_f\in H^0(\P(V),{\cal O}_{\P(V)}(k))$, and the zero locus 
$Z(\sigma_f)\subset\P(V)$
 of this section is a real  hypersurface of degree $k$. We are interested in the number of projective lines contained in this hypersurface $Z(\sigma_f)$. Our method is very natural and starts from the following obvious 
\begin{re} A projective line $\P(W)$ defined by 2-dimensional subspace $W\subset V$ is contained in $Z(\sigma_f)$ if and only if the restriction of   $f$  to $W$ vanishes.
\end{re}

Let $U$ be the tautological  bundle on the Grassmannian $G_2(V)$ of 2-planes in $V$. The form $f$ induces a section $s_f$ in $k$-th symmetric power $S^k(U^\vee)$ of $U^\vee$ with $s_f(W):=\resto{f}{W^k}$.   Hence the projective lines in $Z(\sigma_f)$ correspond bijectively to the zeros of the section $s_f\in \Gamma(G_2(V),
S^k(U^\vee))$. The number of these zeros can be estimated from below using the methods developed in the previous sections provided the following conditions are satisfied
\begin{enumerate}
\item $\rk(S^k(U^\vee)=\dim(G_2(V))$, 
\item $ S^k(U^\vee)$ is relatively orientable.
\end{enumerate}

We have $\rk(S^k(U^\vee)=k+1$ and $\dim(G_2(V))=2(n-2)$, so the first condition is equivalent to   $k=2n-5$.  We will see that, if this holds, then the second condition is always satisfied. More precisely we will see that, when   $k=2n-5$, then $S^k(U^\vee)$ has a {\it canonical}  relative orientation. This has an important consequence:  one can assign to any projective line $l\subset Z(\sigma_f)$,  which is a regular zero of $s_f$, a sign in a completely canonical way.

\subsection{Canonical relative orientations. Elliptic and hyperbolic lines}

We use the notations and constructions introduced above. The problem concerning the relative orientation of the bundle $S^{2n-5}(U^\vee)$ is solved by the following 

\begin{pr} \label{canor} Let $B=G_2(V)$, $\xi=S^{2n-5}(U^\vee)$  and let $\underline{V}$ be the trivial bundle with fibre $V$ on $B$. Then one has a natural identification
$$\oo_B\otimes\oo_\xi=\left[\oo_{U^\vee}^{n^2-4 n+5}\otimes\oo_{\underline{V}}\right]^{\otimes 2} \ ,
$$
so that the sheaf $\oo_B\otimes\oo_\xi$ on $B$  has a canonical trivialization.
\end{pr}

\begin{proof} By Lemma \ref{dettp} below one has a canonical isomorphisms
$$\det(T_B)=\det\left(\Hom\left(U,\qmod{\underline{V}}{U}\right)\right)
=\det\left(U^\vee\otimes \qmod{\underline{V}}{U}\right)=$$
$$=\det(U^\vee)^{n-2}\otimes(\det(\underline{V}))\otimes\det(U)^\vee)^2
= \det(U^\vee)^{n}\otimes \det(\underline{V})^{\otimes 2}\ .$$
In this formula we use  the isomorphism $\det(\underline{V})=\det(U)\otimes\det(\qmod{\underline{V}}{U})$ given by the decomposition $\underline{V}=U\oplus U^\bot$ with respect to an inner product on $V$ (see Remark \ref{W}).
On the other hand, using Lemma \ref{detsym} below, we get
$$\det(S^k(U^\vee))=\det(U^\vee)^{\otimes \frac{k(k+1)}{2}}=\det(U^\vee)^{\otimes  {(2n-5)(n-2)} }\ ,
$$
which proves $\det(T_B)\otimes \det(\xi)=[\det(U^\vee)^{n^2-4 n+5}\otimes \det(\underline{V})]^{\otimes 2}$.
 
\end{proof}

\begin{lm} \label{dettp} For two vector spaces $W'$, $W''$ of dimensions $d'$, $d''$ one has a canonical identification $\det(W'\otimes W'')=\det(W')^{\otimes d''}\otimes(\det W'')^{\otimes d'}$.
\end{lm}
\begin{proof} We use the isomorphism given by
$$(w'_1\otimes w''_1)\wedge \dots \wedge(w'_{d'}\otimes w''_{1})\wedge\dots \wedge (w'_{1}\otimes w''_{d''})\wedge \dots \wedge(w'_{d'}\otimes w''_{d''})\mapsto 
$$
$$\mapsto (w'_1\wedge\dots\wedge  w'_{d'})^{d''}\otimes (w''_1\wedge\dots\wedge  w''_{d''})^{d''} \ .
$$
\end{proof}
\begin{re} \label{orhom} According to the convention used in the proof of the preceding Lemma, the space of matrices $M_{m,n}(\R)=\Hom(\R^n,\R^m)=(\R^n)^\vee\otimes\R^m$ is oriented using the identification $M_{m,n}(\R)\to(\R^n)^m$ which assigns to a matrix the $m$-tuple of its rows. 
\end{re}
\begin{lm}\label{detsym} For a $2$-dimensional vector space $W$ one has a canonical identification $\det(S^k(W))=\det(W)^{\otimes\frac{k(k+1)}{2}    }$.
\end{lm}
\begin{proof} The assignment 
$$(w_1^k)\wedge(w_1^{k-1}  w_2)\wedge\dots\wedge(w_1  w_2^{k-1}) \wedge(w_2^k)\mapsto (w_1\wedge w_2)^{\otimes \frac{k(k+1)}{2}}\  
$$
for $w_1$, $w_2\in W$, extends to a well defined linear isomorphism.
\end{proof}

\begin{re} Although canonical, the isomorphisms used in Lemma  \ref{dettp}, Lemma  \ref{orhom}, and Lemma \ref{detsym} depend on conventions, which must a priori be fixed and used consistently. For instance, the orientation of $M_{m,2}(\R)$ induced by  the isomorphism $M_{m,2}(\R)\to(\R^m)^2$   differs from the orientation used in the proof of Lemma \ref{orhom} by the factor $(-1)^{\frac{m(m-1)}{2}}$. Similarly, the orientation of $S^k(W)$ induced by the isomorphism 
$$(w_2^k)\wedge(w_1  w_2^{k-1})\wedge\dots\wedge(w_1^{k-1}  w_2 ) \wedge(w_1^k)\mapsto (w_1\wedge w_2)^{\otimes \frac{k(k+1)}{2}}\  
$$
 differs from the orientation used in the proof Lemma \ref{detsym} by the factor $(-1)^{\frac{k(k+1)}{2}}$.
\end{re}

\begin{co} The bundle $\xi=S^{2n-5}(U^\vee)$ on the Grassmannian $B=G_2(V)$ is canonically relatively  oriented, i.e., the sheaf $\oo_B\otimes\oo_\xi$ comes with a canonical trivialization $\theta_\mathrm{can}$, which can be defined  choosing   arbitrary orientations of the vector space $V$ and of the 2-planes  $W\subset V$.
\end{co}

\begin{dt}  A projective line  $l=\P(W)\subset Z(\sigma_f)$   which is a regular zero of  the section  $s_f\in\Gamma(G_2(V),S^k(U^\vee))$ will be  called elliptic (hyperbolic) if $\epsilon_{l,\theta_\mathrm{can}}=1$ (respectively -1).
\end{dt}
 
\subsection{The Euler class of the tautological bundle on the Grassmannian}

Let $\tilde G_2(V)$ be  the  Grassmannian of oriented planes of $V$, and $\tilde U$ its tautological rank 2 bundle; note that $\tilde U$  is tautologically oriented. The proof of Proposition \ref{canor} yields a canonical isomorphism
$$\det(T_{\tilde G_2(V)})=\det(\tilde U^\vee)^{n}\otimes \det(\underline{V})^{\otimes 2}\ ,
$$
which shows that $\tilde G_2(V)$ is {\it canonically} oriented.
\begin{pr} \label{powere} $e_{\tilde U^{\oplus(n-2)}}=2[\tilde G_2(V)]'$\ .
\end{pr}
\begin{proof} We endow $V$ with an inner product.  For an oriented plane $W\in \tilde G_2(V)$   we denote by $\bar W$ the same plane endowed with the opposite orientation.

Now fix an orientation of $V$ and  endow the orthogonal complement $W^\bot$ with the  orientation  which makes the isomorphism $W\oplus W^\bot\to V$ orientation preserving.

Let $(w_1,w_2)$, $(w'_1,\dots,w'_{n-2})$ be    bases of $W$, $W^\bot$ compatible with the orientations. We obtain $n-2$ sections $\eta_i$ in $\tilde U$ given by
$$\eta_i(W)=\p_W(w'_i)\ .
$$
The corresponding section $\eta=(\eta_1,\dots,\eta_{n-2})$ of the direct sum $\tilde U^{\oplus (n-2)}$ has exactly two zeros: $W$ and $\bar W$.
We have to compute the signs of the determinant of the intrinsic derivatives $D_W\eta$, $D_{\bar W} \eta$ at these points.

The direct sum decomposition $V=W\oplus W^\bot$ defines a chart 
$$h_W:G_W\map \Hom(W,W^\bot)$$
 of $\tilde G_2(V)$ around $W$ and a trivialization $\tau_W:\resto{U}{G_W}\to G_W\times W$ of $\tilde U$ over $G_W$.  The inverse $h_W^{-1}$ is given by $h_W^{-1}(\phi)=\mathrm{graph}(\phi)$, and the restriction of $\tau_W$ to a fibre $U_{W_1}=W_1$ of $U$ is $\resto{\tau_W}{U_{W_1}}:=\resto{\p_W}{W_1}$.  Therefore, via the chart $h_W$ and the trivialization $\tau_W$, the section $\eta$ corresponds to the map $\hat\eta:\Hom(W,W^\bot)\to W^{\oplus (n-2)}$ given by
$$\hat\eta(\phi)=(\p_W\p_{\mathrm{graph}(\phi)}(w'_1), \dots,\p_W\p_{\mathrm{graph}(\phi)}(w'_{n-2}))\ .
$$

On the other hand, one has the formula (see \cite{DK} p. 8)
$$\p_W\p_{\mathrm{graph}(\phi)} (v)=(1+\phi^*\phi)^{-1}(p_W(v)+\phi^*(p_{W'}(v)))\ ,
$$
which shows that
$$D_W\eta(\phi)=(\phi^*(w'_1),\dots, \phi^*(w'_{n-2}))\ .
$$
Using the bases $(w_1,w_2)$ and $(w_1',\dots w_{n-2}')$ to identify $\Hom(W,W^\bot)$ with the space of matrices $M_{n-2,2}(\R)$ we see that $D_W\eta$ is just the isomorphism $M_{n-2,2}(\R)\to [\R^{2}]^{\oplus (n-2)}$ which maps a matrix  to  the $(n-2)$-tuple of its rows.

\end{proof}

\subsection{The Euler class computation}

The purpose of this section is the computation of the Euler class 
$$e_{S^k(U)}\in H^{2(n-2)}(G_2(V),\oo_{S^k(U)})\ .$$

Note that we have natural isomorphisms 
$$H^{2(n-2)}(G_2(V),\oo_{S^k(U)})\simeq  H^{2(n-2)}(G_2(V),\oo_{G_2(V)})\simeq\Z\ ,$$
where the first isomorphism is given by the canonical trivialisation $\theta_\mathrm{can}$ introduced in the previous section, and the second is induced by Poincaré duality.
The main result is
\begin{thry} \label{esym} Let $V$ be an $n$-dimensional real vector space, and let $U$, $\tilde U$  be the tautological bundles on the Grassmannians $G_2(V)$, $\tilde G_2(V)$. Put $k:=2n-5$. Then
\begin{equation}\label{esymh} 
e_{S^{k}(\tilde U)}=2 \Pi_{j=0}^{n-3} (2j+1)=2( 1\cdot 3\cdot\hdots\cdot k)\ ,
\end{equation}
\begin{equation}\label{esymb}
e_{S^{k}(U)}= \Pi_{j=0}^{n-3} (2j+1)= 1\cdot 3\cdot\hdots\cdot k\ .
\end{equation}
 
\end{thry}

 In order to prove this result we need a preparation. Let $W$ be real plane endowed with a complex structure $J\in\End(W)$, and let $s:=n-3$, i.e., $k=2s+1$. The symmetric power ${\cal S}:=S^k(W)$ comes with a natural complex structure ${\cal J}$ defined on symmetric monomials  by ${\cal J} (w_1\vee\dots \vee w_k):={\cal J} (w_1)\vee\dots \vee {\cal J}(w_k)$. We want a simple description of the complex vector space $({\cal S},{\cal J})$ in terms of the complex line $L:=(W,J)$.

The complexification $W^\C$ decomposes as $W^\C=W^{10}\oplus W^{01}$, where the two summands are the $i$, respectively $-i$ eigenspaces of the $\C$-linear extension $J^\C$ of $\C$.

Accordingly, we get a decomposition
$$S^k(W)\otimes\C=S^k(W^\C)=\bigoplus_{j=0}^k  (W^{10})^{\otimes(k-j)}\otimes  (W^{01})^{\otimes j}\ .
$$
Note that ${\cal J}^\C$ acts on the summand $(W^{10})^{\otimes(k-j)}\otimes  (W^{01})^{\otimes j}$ by multiplication with $i^{k-j}(- i)^j=(-1)^j i^{k}=(-1)^{j+s} i$.

This shows that 
$${\cal S}^{10}=\left\{\begin{array}{ccc}
\bigoplus_{l=0}^s (W^{10})^{\otimes(k-2l)}\otimes  (W^{01})^{\otimes 2l}&\rm when& s\hbox{ is even}\\
\bigoplus_{l=0}^s (W^{10})^{\otimes(k-(2l+1))}\otimes  (W^{01})^{\otimes (2l+1)}&\rm when& s\hbox{ is odd}
\end{array} \right. \ .
$$

Since the projections $({\cal S},{\cal J})\to {\cal S}^{10}$, $({\cal S},-{\cal J})\to {\cal S}^{01}$, $L=(W,J)\to W^{10}$, $\bar L\to W^{0,1}$  are   isomorphisms of $\C$-vector spaces, and taking into account that $\dim_\C({\cal S},{\cal J})=\frac{k+1}{2}$ is even when $s$ is odd, we obtain
\begin{lm} With the notations and definitions above there is a canonical $\R$-linear isomorphism 
$$({\cal S},{\cal J})  \textmap{\simeq} 
\bigoplus_{l=0}^s  L^{\otimes(k-2l)}\otimes  \bar L^{\otimes 2l} 
$$ 
which is $\C$-linear when $s$ is even and $\C$-anti-linear when $s$ is odd. In both cases this canonical isomorphism is orientation preserving with respect to the complex orientation of ${\cal S}$.
\end{lm}

Using a $J$-Hermitian metric $h$ on $W$ we obtain an identification $\bar L=L^\vee$, hence

\begin{lm} \label{orwithJ} Let $L=(W,J,h)$ be a Hermitian line. There exists a  canonical $\R$-linear isomorphism 
$$({\cal S},{\cal J})  \textmap{\simeq} 
\bigoplus_{l=0}^s  L^{\otimes(k-4l)}   
$$ 
which is $\C$-linear when $s$ is even and $\C$-anti-linear when $s$ is odd. In both cases this canonical isomorphism is orientation preserving with respect to the complex orientation of ${\cal S}$.
\end{lm} 

\begin{re} The complex  orientation of the symmetric power ${\cal S}=S^k(W)$ defined by ${\cal J}$ differs from the canonical orientation used in the proof of   Lemma  \ref{detsym}  by the factor $(-1)^{\frac{s(s+1)}{2}}$.
\end{re} 
\begin{proof} Let $(w_1,w_2)$ be a basis of $W$ with $w_2=Jw_1$. The canonical orientation of  $S^k(W)$ used in the proof of Lemma  \ref{detsym} is defined by the basis %
$$(w_1^k,w_1^{k-1}w_2,\dots,w_1w_2^{k-1},w_2^k)\ ,$$ whereas the complex orientation is defined by the basis 
$$(w_1^k,{\cal J}w_1^k,w_1^{k-1}w_2, {\cal J}(w_1^{k-1}w_2),\dots,w_1^{k-s}w_2^s,{\cal J}(w_1^{k-s}w_2^s))\ .$$
It suffices to note that ${\cal J}(w_1^{k-i}w_2^i)=(-1)^iw_1^i w_2^{k-i}$ and that the permutation
$$\left(\begin{array}{ccccccccc}
0&1&2&3&4&5&\dots&k-1&k\\
0&k&1&k-1&2&k-2&\dots&s&s+1
\end{array}\right)
$$
is always even, so the two orientations differ by $(-1)^{\sum_{i=0}^s i}$.
\end{proof}

Taking this factor into account we obtain:

\begin{re} \label{coco} Let $L=(W,J,h)$ be a Hermitian line. If we endow $S^k(W)$ with the canonical orientation  used in  Lemma  \ref{detsym}, the isomorphism given by Lemma \ref{orwithJ} changes the orientation by the factor $(-1)^{\frac{s(s+1)}{2}}$.
\end{re}
{\ }\\
{\it Proof} of Theorem \ref{esym}. We choose an inner product on $V$, and obtain an induced inner product on every oriented plane $W\in \tilde G_2(V)$.  In this way the tautological bundle $\tilde U$ becomes an $\SO(2)$ bundle, so it defines a Hermitian line bundle, which we  denote by $\tilde \Lambda$.

Using fibrewise the canonical isomorphism given by Lemma \ref{orwithJ}, and taking into account Remark \ref{coco}, we obtain a  bundle isomorphism
$$S^k(\tilde U)\textmap{\simeq}  
\bigoplus_{l=0}^s  \tilde\Lambda^{\otimes(k-4l)}  
$$ 
which  multiplies the orientation by the factor $(-1)^{\frac{s(s+1)}{2}}$.  Since $c_1(\tilde \Lambda) =e(\tilde U)$ we can apply Proposition \ref{powere} and get
\begin{equation}\label{ee}
e_{S^k(\tilde U)}= (-1)^{\frac{s(s+1)}{2}} 
2\prod_{l=0}^s  (k-4l)  \ .
\end{equation}

Note that the set  
$\{| k-4l|\ |\ 0\leq l\leq s\}$
 coincide  with the set $\{1,3,\dots,k\}$ of odd numbers between $1$ and $k$. On the other hand, the number of negative factors in the product on the right in (\ref{ee}) is
$$\left\{
\begin{array}{ccc}  {\frac{s}{2}} &\rm when & s \hbox{ is even}\\ 
 {\frac{s+1}{2}}&\rm when& s\hbox{ is odd}\end{array}\right. \ .
$$

But this number has the same parity as $\frac{s(s+1)}{2}$.
This proves the first formula.  The second formula is proved using the double cover $c:\tilde G_2(V)\to G_2(V)$, the functoriality of the Euler class (formula (\ref{func})),  and the obvious equality $c^*([G_2(V)]')=2[\tilde G_2(V)]'$ where $[G_2(V)]'$, $[\tilde G_2(V)]'$ are the canonical generators of the cyclic groups $H^{2(n-2)}(G_2(V),\oo_{S^k(U)})$, $H^{2(n-2)}(\tilde G_2(V),\Z)$.
\qed
\vspace{3mm}

By Remark $\ref{dual}$ one has $e_{S^k(U^\vee)}=e_{S^k(U)}$. Taking into account Proposition \ref{p2} we obtain the following lower bound for the number of lines on real hypersurfaces:

\begin{co} Let $V$ be an $n$-dimensional real vector space, $k:=2n-5$, $f\in S^k(V^\vee)$, and let $Z(\sigma_f)$ be the corresponding real hypersurface  in $\P(V)$. Suppose that the set ${\cal R}$ of  lines in $Z(\sigma_f)$ is finite. Then, denoting by $\mathrm{mult}_l(s_f)$ the  multiplicity of a real line $l\in{\cal R}$ regarded as   zero of the section $s_f$, we have the estimate
$$\sum_{l\in {\cal R}} \mathrm{mult}_l(s_k)\geq 1 \cdot 3\cdot\hdots\cdot  k \ .$$
\end{co}

\section{Computation of the intrinsic derivative }

Let $V$ be a real $n$-dimensional vector space, and for $m\leq n$ let $G_m(V)$  be the Grassmannian of $m$-dimensional linear subspaces of $V$. An element $f\in S^k(V^\vee)$ defines a section $s_f\in\Gamma(G_m(V),S^k(U^\vee))$, where $U$ denotes the tautological bundle on $G_m(V)$. The value $s_f(W)$ at  a point $W\in G_m(V)$  is just the restriction   $\resto{f}{W^k}$.

Equivalently, we can identify the space $S^k(V^\vee)$ (the bundle $S^k(U^\vee)$)  with the space $\R[V]_k$ (the bundle $\R[U]_k$) of homogeneous polynomials of degree $k$ on the space $V$ (the bundle $U$).
Let $q\in \R[V]_k$ be the homogeneous polynomial corresponding to $f$.
We will compute the intrinsic derivative 
$$
D_W(s_f):T_W(G_m(V))=\Hom(W, {V}/{W})\map \R[W]_k
$$
of this section at a point $W$ in the zero locus $Z(s_f)\subset G_m(V)$.

Let $W'$ be a complement of $W$ in $V$. The direct sum decompostion $V=W\oplus W'$ defines a chart $h_{W,W'}:G_{W,W'}\to \Hom(W,W')$ of $G_m(V)$ around $W$, whose inverse is the map $h_{W,W'}^{-1}:\Hom(W,W')\to G_{W,W'}$ sending a linear map $\phi:W\to W'$ to its graph $\Gamma_\phi\subset W\times W'\simeq W\oplus W'=V$.
The  direct sum decomposition $V=W\oplus W'$ induces a   trivialization 
$$\tau_{W,W'}:\resto{U}{G_{W,W'}}\to G_{W,W'}\times W 
$$
of the bundle $U$ on the open set  $G_{W,W'}\subset G_m(V)$, which is defined by
$$\tau_{W,W'}(T,t):=\mathrm{pr}_W(t)\ .
$$
Here $T\in G_{W,W'}$, $t\in T=U_T$ and $\mathrm{pr}_W:V\to W$ stands for the projection on $W$ with respect to the direct sum decompostion $V=W\oplus W'$. Note that the restriction of this local trivialization to the fibre $U_W=W$ is the identity.

If $T=h_{W,W'}^{-1}(\phi)$, i.e., $T$ is the graph of $\phi$, then  this trivialization identifies the fibre $U_T=T$ with $U_W=W$ via the isomorphism 
$$T\ni t\mapsto \mathrm{pr}_W(t)\in W\ .
$$

Note that the inverse of this isomorphism is $W\ni w\mapsto w+\phi(w)\in T$. Using the chart $h_{W,W'}$ and the local trivialization $\tau_{W,W'}$ we see that the section $s_f$ is  defined around the point $W\in G_m(V)$  by the map 
$\hat s_{f,W,W'}:\Hom(W,W')\to \R[W]_k$
 given by
$$\hat s_{f,W,W'}(\phi)(w):=q(w+\phi(w))\ .
$$
The intrinsic derivative $D_W(s_f)$ of $s_f$ at $W\in Z(s_f)$ can be identified with the derivative at $0\in\Hom(W,W')$ of $\hat s_{f,W,W'}$, i.e., with the linear map
$$\sigma_{f,W,W'}:\Hom(W,W')\to \R[W]_k \ ,
$$
defined by 
\begin{equation}\label{dif} \sigma_{f,W,W'}(\phi)(w)=d_w(q)(\phi(w))\ ,
\end{equation}
where $d_w(q)\in \Hom(V,\R)$ is the differential of $q$ at $w$.  Now define a linear map 
$$\fg:V/W\to \R[W]_{k-1}$$ by 
$$\fg([v])(w):=f(v,\underbrace{w,\dots,w}_\textrm{$k-1$})\ .
$$
Then one has  $d_w(q)(v)=k\fg([v])(w)$. With these notations  we have proved the following
\begin{lm} The intrinsic derivative $D_W(s_f)$ of $s_f$ at $W\in Z(s_f)$ is the map $D_W(s_f):\Hom(W,V/W)\to \R[W]_k$  given by
\begin{equation}\label{conclusion}D_W(s_f)(\phi)(w)= k\fg(\phi(w))(w)\ .
\end{equation}
\end{lm}
%{\bf Remark:}  {\it Meine Abbildung $\fg_1$ stimmt mit Deiner Abbildung $p(f)$ überein.}
%

%
\vspace{4mm}
\paragraph{\bf The determinant of $D_W(s_f)$}
Let $W\in G_2(V)$ be a zero of $s_f$, and let  $W'$ be a complement  of $W$ in $V$. Fix bases $(w_1,w_2)$, $(w_1',\dots,w'_{n-2})$ in $W$ and $W'$ respectively. Using these bases we identify $\Hom(W,W')$ with $M_{n-2,2}(\R)$,   and consider the basis  $(E_{ij})_{\substack{1\leq i\leq n-2\\1\leq j\leq 2}}$. Then a matrix $M\in M_{n-2,2}(\R)$ decomposes as 
$$M =\sum_{\substack{1\leq i\leq n-2\\ 1\leq j\leq 2}} m_{ij} E_{ij}\ .$$
 Using the convention explained in Remark \ref{orhom}, the orientation induced by the two bases $(w_1,w_2)$, $(w'_1,\dots,w'_{n-2})$ on $\Hom(W,W')$ is  defined by the ordered basis $(E_{11},E_{12},E_{21},E_{22},\dots,E_{n-2,1},E_{n-2,2})$. 

The linear map $\fg$ can be regarded as   a system $(Q_1,\dots,Q_{n-2})$ of  $n-2$ homogeneous polynomials of degree $k-1$ on $W$.  Let $x_i$ be the coordinates induced by the basis $(w_1,w_2)$. Writing $Q_i(x_1,x_2)=\sum_{l=0}^{k-1} a_{l,i} x_1^{k-1-l} x_2^l$ we see that the matrix $F\in M_{k,n-2}(\R)$ associated with $\fg$ with respect to our choices of bases is
$$F=\left(\begin{matrix}
a_{0,1}&a_{0,2}&\hdots&a_{0,n-2}\\
a_{1,1}&a_{1,2}&\hdots& a_{1,n-2}\\
a_{2,1}&a_{2,2}&\hdots&a_{2,n-2} \\
\vdots &\vdots&\hdots&\vdots \\
a_{k-1,1}&a_{k-1,2}&\hdots&a_{k-1,n-2} \\
\end{matrix}
\right) \ .
$$

We are interested in the determinant of the linear  map
$$\sigma:M_{n-2,2}(\R)\map \R[x_1,x_2]_k\ ,
$$ 
induced by $\frac{1}{k}\sigma_{f,W,W'}$ via the identifications explained above. One has
$$\sigma(M)=  (Q_1(x_1,x_2),\dots,Q_{n-2}(x_1,x_2))M\left(\begin{array}{c} x_1\\ x_2\end{array}\right)\ .$$
In particular
$ \sigma(E_{ij})=  Q_i x_j\ ,
$
so that
$$\sigma(E_{i1})=\sum_{l=0}^{k-1} a_{l,i} x_1^{k-l} x_2^l  \ ,\ \sigma(E_{i2})=\sum_{l=0}^{k-1} a_{l,i} x_1^{k-l-1} x_2^{l+1}\ .
$$

Now write  $\sigma(E_{i1})=\sum b_{l,i} x_1^{k-l} x_2^l$ and  $\sigma(E_{i2})=\sum c_{l,i} x_1^{k-l} x_2^l$. Then we get the following formulae which give the columns of the matrix $\Sigma$ representing  $\sigma$ with respect to the  bases $(E_{11},E_{12},E_{21},E_{22},\dots,E_{n-2,1},E_{n-2,2})$, $(x_1^k,x_1^{k-1}x_2,\dots,x_2^k)$:
$$b_{l,i}=a_{l,i} \hbox { for }0\leq l\leq k-1\ ,\ b_{k,i}=0\ ,
$$
$$c_{l,i}=a_{l-1,i} \hbox{ for } 1\leq l\leq k\ ,\ c_{0,i}=0\ .
$$
Therefore we find:
\begin{equation}\label{Sigma}\Sigma=\left(\begin{matrix}
a_{0,1}&0&\hdots&a_{0,n-2}&0\\
a_{1,1}&a_{0,1}&\hdots& a_{1,n-2}&a_{0,n-2}\\
a_{2,1}&a_{1,1}&\hdots&a_{2,n-2}& a_{1,n-2}\\
\vdots &\vdots&\hdots&\vdots&\vdots\\
a_{k-1,1}&a_{k-2,1}&\hdots&a_{k-1,n-2}&a_{k-2,n-2}\\
0&a_{k-1,1}&\hdots&0&a_{k-1,n-2}
\end{matrix}
\right) \  
\end{equation}
\\ \\
{\bf Example} ( $n=4$, $k=3$): In this case the matrix  of $\sigma$ is
\begin{equation}\label{SigmaEx}\Sigma=\left(\begin{matrix}
a_{0,1}&0&a_{0,2}&0\\
a_{1,1}&a_{0,1}& a_{1,2}&a_{0,2}\\
a_{2,1}&a_{1,1}&a_{2,2}& a_{1,2}\\
0&a_{2,1}&0&a_{2,2}
\end{matrix}
\right)  \ .
\end{equation}

Developing the determinant of $\Sigma$ with respect to the columns (1,3) we get
$$\det(\sigma)=-\left|\begin{matrix} a_{0,1}& a_{0,2}\\
a_{1,1}& a_{1,2}\end{matrix} \right| \left|\begin{matrix} a_{1,1}& a_{1,2}
\\ a_{2,1}& a_{2,2}\end{matrix} \right|+
\left|\begin{matrix} a_{0,1}& a_{0,2}\\
a_{2,1}& a_{2,2}\end{matrix} \right|^2  \ .
$$
Note that $\fg$ can also be regarded as a homogeneous polynomial of degree $k-1$ in the variables $x_1$, $x_2$ with coefficients in $\Hom(V/W,\R)$. Let $(u_1,u_2)$ be the coordinates in $V/W$ corresponding to the basis $(\bar w_1',\bar w_2')$. Then
 
$$\fg(x_1,x_2)=(a_{0,1}u_1 +a_{0,2} u_2)x_1^2+(a_{1,1}u_1+a_{1,2} u_2) x_1 x_2+ (a_{2,1}u_1+a_{2,2}u_2)x_2^2\ .
$$
 A simple  computation shows that 
$$\det(\sigma)=\frac{1}{16}\Delta_u(\Delta_x(\fg))\ ,
$$
where $\Delta_x$ denotes the discriminant of $\fg$ regarded as quadratic form in the variables $x_1$, $x_2$. This determinant is a quadratic form  in the variables $u_1$, $u_2$, and $\Delta_u$ stands for the discriminant of this quadratic form.
\begin{re} The matrix (\ref{SigmaEx}) is precisely the Sylvester matrix of the two polynomials $Q_1$, $Q_2$, so in the case $n=4$ the determinant $\det(\sigma)$ coincides with the resultant of these polynomials.
\end{re}
In the general case, the determinant $\det(\Sigma)$
 can be developed with respect to the system of  columns $(1,3,\dots,k)$. For a system $I=(i_1,i_2,\dots,i_{n-2})$ with $1\leq i_1<i_2<\dots<i_{n-2}\leq k$, let $m_{I}$ be the   minor of $F$ formed with  the corresponding   rows. Let  $\bar I=(j_1,j_2,\dots,j_{n-3})$ be    the complement of $(i_1,i_2,\dots,i_{n-2})$ in $\{1,2,\dots,k\}$, ordered in the obvious way. If $j_1\geq 2$ (or, equivalently, if $i_1=1 $) we define $\tilde I\subset \{1,2,\dots,k\}$ of length $n-2$ by
$$\tilde I:=(j_1-1,j_2-1,\dots,j_{n-3}-1,k)\ .
$$
Note that $m_I$ coincides with the minor of $\Sigma$ which is associated with the system of rows $(i_1,i_2,\dots,i_{n-2})$  and the fixed columns $(1,3,\dots,k)$, and that the complementary  minor in $\Sigma$  vanishes when $j_1=1$ and coincides with $m_{\tilde I}$ when $j_1\geq 2$.  Therefore
\begin{equation}\label{detformula} \det(\Sigma)= \sum_{\substack{I\subset \{1,\dots,k\}\\| I|=n-2,\ i_1=1}} (-1)^{\sum_{j=0}^{n-3} (2j+1)+ \sum_{l=1}^{n-2} i_l}\   m_{\raisebox{-0.35ex}{$\scriptstyle I$}} m_{\tilde I} $$
$$=\sum_{\substack{I\subset \{1,\dots,k\}\\| I|=n-2,\ i_1=1}} (-1)^{n+ \sum_{l=1}^{n-2} i_l}\  m_{\raisebox{-0.35ex}{$\scriptstyle I$}} m_{\tilde I} \ .
\end{equation}
We can interpret  this formula geometrically in the following way: 

Let $W$ be real 2-dimensional vector space, and $s\in\N$.\\

\begin{dt} An $s$-plane $I\in G_s(S^{2s-2}(W^\vee))$ will be called regular if the composition
$$\delta_I:I\otimes W^\vee \textmap{i\otimes\id_{W^\vee}} S^{2s-2}(W^\vee)\otimes W^\vee\textmap{m_W}  S^{2s-1}(W^\vee)
$$
is an isomorphism, where $i:W\to S^{2s-2}(W^\vee)$ denotes the inclusion of $I$ in $S^{2s-2}(W^\vee)$, and $m_W$ is the  epimorphism defined by  multiplication of homogeneous polynomials of degree $(2s-2)$ with linear forms.
\end{dt}

If $(f_1,\dots,f_s)$ is a basis of $I$ and $(x_1,x_2)$ a system of linear coordinates on $W$, we can write
$$f_i(x_1,x_2)=\sum_{l=0}^{2s-2} a_{l,i} x_1^{2s-2-l} x_2^l\in S^{2s-2}(W^\vee)\ ,
$$
and the matrix of $\delta_I$ with respect to the basis $(f_i\otimes x_j)_{\substack{1\leq i\leq s\\ 1\leq j\leq 2}}$ of $I\otimes W^\vee$ and the standard basis of $S^{2s-2}(W^\vee)$ is precisely the matrix $\Sigma$ written above.
Denote by $T$ the tautological $s$-bundle on  $G_s(S^{2s-2}(W^\vee)$). The family of morphisms $(\delta_I)_{I\in G_s(S^{2s-2}(W^\vee))}$ defines a bundle morphism
$$\delta: T\otimes \underline{W}^\vee\to S^{2s-1}(\underline{W}^\vee)\ ,
$$
where $\underline{W}$ denotes the trivial 2-bundle  with fibre $W$. The determinant $\det(\delta)$ of this bundle morphism  can be regarded as a section
$$\det(\delta)\in\Gamma(G_s(S^{2s-2}(W^\vee),\det(T^\vee)^{\otimes 2}\otimes\det(\underline{W}^\vee)^{\otimes -s}\otimes \det(S^{2s-1}(\underline{W}^\vee)))=
$$
$$=\Gamma(G_s(S^{2s-2}(W^\vee)),\det(T^\vee)^{\otimes 2}\otimes\det(\underline{W}^\vee)^{\otimes 2s} )\ .
$$

Note that the  real line bundle $\det(T^\vee)^{\otimes 2}\otimes\det(\underline{W}^\vee)^{\otimes 2s}$ is {\it canonically oriented} (using arbitrary fibrewise orientations of $T$ and $\underline{W}$), so that one can assign  a sign $\epsilon_I$ to every regular $s$-plane $I$ in an intrinsic way. This sign coincides with the sign of the determinant of the matrix $\Sigma$ associated with $I$ and the basis $(f_1,\dots,f_s)$, and  has a clear geometric interpretation. Consider the Plücker embedding 
$$\pi: G_s(S^{2s-2}(W^\vee))\to \P(E)\ ,$$
 where $E:=\wedge^s(S^{2s-2}(W^\vee))$.
The section $\det(\delta)$ is the pull-back of  a section $\rho$ in the line bundle $${\cal O}_{\P(E)} (2)\otimes \det(\underline{W}^\vee)^{\otimes 2s}$$
on the projective space $\P(E)$, which defines a quadric $Q\subset\P(E)$. Putting $k:=2s-2$, using the basis 
$$\left((x_1^{k-1-i_1} x_2^{i_1})\wedge \dots\wedge (x_1^{k-1-i_{s}} x_2^{i_{n-2}})\right)_{\substack{I\subset \{1,\dots,k\}\\| I|=s}}\ ,$$
and the corresponding system of linear coordinates $(\mu_I)_{\substack{I\subset \{1,\dots,k\}\\| I|=s}}$ of $E$,  the computation of $\det(\Sigma)$ above shows that $Q$ is defined by the equation 
$$\sum_{\substack{I\subset \{1,\dots,k\}\\| I|=n-2,\ i_1=1}} (-1)^{n+ \sum_{l=1}^{n-2} i_l}\   \mu_{\raisebox{-0.35ex}{$\scriptstyle I$}} \mu_{\tilde I}=0\ .
$$ 
The   complement of $Q$ in $\P(E)$ has two connected components:
$$\P(E)_\pm:=\{[\nu]\in \P(E)\ |\ \ \pm\hspace{-5mm}\sum_{\substack{I\subset \{1,\dots,2s-2\}\\| I|=s,\ i_1=1}} (-1)^{n+ \sum_{l=1}^{n-2} i_l}\  \mu_{\raisebox{-0.35ex}{$\scriptstyle I$}}(\nu) \mu_{\tilde I}(\nu)>0\}\ .
$$ 
For a regular $s$-plane $I$ we see that the sign $\epsilon_I$ is determined by the position of $\pi(I)$ with respect to this quadric. 
\\

We come back to the determinant of the intrinsic derivative $D_W(s_f)$. Take now $s=n-2$, and note that, with the notations introduced above,  one has  
$$m_I=\mu_I((\wedge^{n-2}\fg)(w_1'\wedge\dots\wedge w'_{n-2}))\ .$$  
Formula (\ref{detformula}) shows now that 
\begin{pr} The intrinsic derivative $D_W(s_f)$ is an isomorphism if and only if the linear map $\fg:V/W\to \R[W]_{k-1}=S^{k-1}(W^\vee)$ is injective and its image $I:=\im(\fg)$ is a regular $n-2$-plane.   If this is the case, the sign of $\det(D_W(s_f))$ with respect to the canonical orientations introduced above is determined by the component  of $\P(\wedge^{n-2}(S^{k-1}(W^\vee)))\setminus Q$ to which $\pi(I)$ belongs.
\end{pr}

\end{document}